\documentclass[
final, nomarks
]{dmtcs-episciences}
\usepackage{chngcntr}
\counterwithin{figure}{section}

\usepackage[usenames]{color}
\usepackage{amsmath,amsthm,amssymb,url,graphicx,tikz,multirow}
\theoremstyle{plain}
\newtheorem{theorem}{Theorem}

\newtheorem{lemma}{Lemma}
\newtheorem{prop}{Proposition}
\newtheorem{corollary}{Corollary}

\theoremstyle{definition}

\usepackage{enumerate}
\usepackage{color}
\usepackage{geometry}
\usepackage{float}

\usepackage[utf8]{inputenc}
\usepackage{subfigure}
\usepackage{tikz}

\author{Mitchell Paukner\thanks{Student Blugold Commitment Differential Tuition funds through the University of Wisconsin-Eau Claire Summer Research Experiences for Undergraduates }
\and
Lucy Pepin\thanks{Student Blugold Commitment Differential Tuition funds through the University of Wisconsin-Eau Claire Summer Research Experiences for Undergraduates }
\and
Manda Riehl\thanks{University of Wisconsin - Eau Claire Office of Research and Sponsored Programs}
\and
Jarred Wieser \thanks{University of Wisconsin - Eau Claire Ronald E. McNair Postbaccalaureate Achievement Program}}

\title{Pattern Avoidance in Task-Precedence Posets}
\affiliation{
  Department of Mathematics\\
University of Wisconsin -- Eau Claire\\
Eau Claire, WI, USA \\}
\keywords{permutation pattern, poset}
\received{2015-11-3}
\accepted{2016-1-15}
\begin{document}
\publicationdetails{18}{2016}{2}{3}{1324}
\maketitle

\begin{abstract}
We have extended classical pattern avoidance to a new structure: multiple task-precedence posets whose Hasse diagrams have three levels, which we will call diamonds. The vertices of each diamond are assigned labels which are compatible with the poset. A corresponding permutation is formed by reading these labels by increasing levels, and then from left to right. We used Sage to form enumerative conjectures for the associated permutations avoiding collections of patterns of length three, which we then proved. We have discovered a bijection between diamonds avoiding 132 and certain generalized Dyck paths. We have also found the generating function for descents, and therefore the number of avoiders, in these permutations for the majority of collections of patterns of length three. An interesting application of this work (and the motivating example) can be found when task-precedence posets represent warehouse package fulfillment by robots, in which case avoidance of both 231 and 321 ensures we never stack two heavier packages on top of a lighter package.
\end{abstract}

\section{Introduction}\label{Introduction}

In this paper, we continue a rich tradition of extending the notion of classical pattern avoidance in permutations to other structures. Given permutations $\pi=\pi_1\pi_2 \cdots \pi_n$ and $\rho = \rho_1 \rho_2 \cdots \rho_m$ we say that $\pi$ \emph{contains} $\rho$ as a pattern if there exist $1 \leq i_1 < i_2 < \cdots < i_m \leq n$ such that $\pi_{i_a} < \pi_{i_b}$ if and only if $\rho_a < \rho_b$. In this case we say that $\pi_{i_1}\pi_{i_2}\cdots \pi_{i_m}$ is \emph{order-isomorphic} to $\rho$ and that $\pi_{i_1}\pi_{i_2}\cdots \pi_{i_m}$ \emph{reduces to} $\rho$. If $\pi$ does not contain $\rho$, then $\pi$ is said to \emph{avoid} $\rho$. The classical definition of pattern avoidance in permutations has shown itself to be worthwhile in many fields including algebraic geometry \cite{Henning} and theoretical computer science \cite{Knuth}. Analogues of pattern avoidance have been developed for a variety of combinatorial objects including Dyck paths \cite{Dyck13}, tableaux \cite{Lewis}, set partitions \cite{sagan}, trees \cite{Rowland10}, posets \cite{Hopkins}, and many more. We use a definition of pattern avoidance that is similar to that used in the study of heaps \cite{Heaps}, but distinct from that used in previous studies of trees. Unlike the question studied by Hopkins and Weiler \cite{Hopkins} which identified classes of posets for which certain properties are preserved, we extend the enumerative question of pattern avoidance to a particular class of posets.

A \emph{task-precedence poset} is a poset which represents the order relations between several tasks to be completed. We are particularly interested in considering $d$ identical task-precedence posets, and here we focus our attention on those sets of tasks that require one task be completed before any others, and one final task after any others, with no restrictions on the rest of the tasks in the list. When considering a list of $4$ tasks, the Hasse diagram of this poset is a diamond, and as such we will refer to a task-precedence poset of this type with $v$ tasks as a \emph{diamond} with $v$ vertices (each with $v-2$ vertices in the middle level). We then assign unique labels from $\{1, 2, \ldots, vd\}$ to each vertex such that the labels obey the order relations of each diamond. We then refer to the set of all such labelled collections of diamonds as $\mathcal{D}_{v,d}.$

Given an element $D$ of $\mathcal{D}_{v,d}$ we associate a permutation $\pi_{D}$ by recording the vertex labels as they are encountered reading the labels on each diamond consecutively, left to right by levels, beginning with the least element. For example, if $D$ is as pictured in Figure \ref{tikz1}, then $\pi_{D}=156273498(10).$ We say that $D$ contains (respectively avoids) $\rho$ as a pattern if $\pi_{D}$ contains (respectively avoids) $\rho$ as a classical pattern, using the definition above. We will abuse notation and sometimes refer to an element of $\mathcal{D}_{v,d}$  and it’s associated permutation interchangeably. Let $\mathcal{D}_{v,d}(P)$ be the elements of $\mathcal{D}_{v,d}$ that avoid all patterns in list $P.$ While Figure \ref{tikz1} contains $123, 132, 213, 231, 312$, it is a member of $\mathcal{D}_{5,2}(321)$. Two patterns on diamonds, $\alpha$ and $\beta$, are said to be \emph{$d$-Wilf-equivalent} if they have the same enumeration, that is, if $\left|\mathcal{D}_{v,d}(\alpha)\right|= \left|\mathcal{D}_{v,d}(\beta)\right|$ for all $v$ and $d$. If so, we write $\alpha \sim_W \beta$.

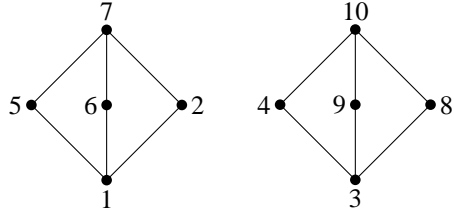
\begin{figure}
\begin{center}
\begin{tikzpicture}

	\coordinate[label=below:1] (a) at (2,-1);
	\coordinate[label=left:5] (b) at (1,0);
	\coordinate[label=above:7] (c) at (2,1);
	\coordinate[label=right:2] (d) at (3,0);
    \coordinate[label=left:6] (e) at (2,0);

	\draw (a) -- (b) -- (c) -- (d) -- (a);
    \draw (a) -- (e) -- (c);

	\foreach \i in {a,b,c,d,e}
		\fill (\i) circle (2pt);
\end{tikzpicture}
\hspace{.3cm}
\begin{tikzpicture}

	\coordinate[label=below:3] (a) at (2,-1);
	\coordinate[label=left:4] (b) at (1,0);
	\coordinate[label=above:10] (c) at (2,1);
	\coordinate[label=right:8] (d) at (3,0);
    \coordinate[label=left:9] (e) at (2,0);

	\draw (a) -- (b) -- (c) -- (d) -- (a);
    \draw (a) -- (e) -- (c);

	\foreach \i in {a,b,c,d,e}
		\fill (\i) circle (2pt);
\end{tikzpicture}
\end{center}\caption{An element of $\mathcal{D}_{5,2}(321)$.}\label{tikz1}
\end{figure}

Our motivation comes from a real-life application, namely a fleet of robots all completing the same sequence of tasks in a warehouse for package fulfillment. In 2011, instead of having human employees walk the warehouse floor retrieving items one after another to complete an order, Amazon began utilizing Kiva robots in their package fulfillment warehouses \cite{Kiva}. Each robot executes $4$ pieces of the larger task. We assign robots to diamonds ordered by the weight of the object they will deliver, heaviest object first, so that the tasks to retrieve the first, heaviest object are represented in diamond $1$, and the lightest object by the final diamond. First the robot drives to the appropriate inventory rack and mounts the rack on its back. Then it can either drive through the warehouse highways to its picker (the human employee who will retrieve the item off the rack without leaving their station), or it can rotate itself so that the appropriate side of the rack is facing the picker. Both of these need to be completed before the final step: having the item picked off the rack by the human employee in order to place it in its shipping box. In this way, completing one order of $d$ items from Amazon.com is exactly the task-precedence poset represented by $d$ diamonds with $4$ vertices each.

We now give an example of this process, referring throughout to Figure \ref{robot}. A customer has made an order for $3$ objects, $o_1, o_2, o_3$, with weights $w(o_1) > w(o_2)> w(o_3).$ Thus the leftmost diamond will represent the tasks completed by the robot retrieving object $1$, the center diamond for retrieving object $2$, and the rightmost diamond for retrieving object $3$. The labels represent the order in which each task of the $12$ total tasks is executed. Each robot operates autonomously and independently, and each faces its own challenges. For example one of the objects may be at the back of the warehouse, there may be significant traffic along some of the paths the robots travel through the warehouse, or the robot assigned to retrieve an object may still be executing its previous assignment. Thus the labels on the least elements of each diamond can vary significantly, and there can be a large difference in the labelling of the least element of a particular diamond and its greatest element. In Figure \ref{robot}, the first task completed is that the robot for object $3$ arrives and picks up the rack containing object $3$. Next, the robot retrieving object $1$ arrives at the rack containing object $1$. Next, the robot carrying object $1$ rotates its rack on its back to have the correct orientation to the picker. This continues, and based on the labelling of the elements, we see that object $3$ (the lightest) is placed in its shipping box first (in step $9$), then object $1$ (in step $11$), and then object $2$ (in step $12$). So our human picker has placed two heavier objects on a lighter object (unless they rearrange the objects after packing). Then a a sufficient (though not necessary) condition to ensure that two heavier objects do not arrive after a lighter object is that the associated permutation avoid $231$ and $321$.

\begin{figure}[H]
\begin{center}
\begin{tikzpicture}

	\coordinate[label=below:2] (a) at (2,-1);
	\coordinate[label=left:10] (b) at (1,0);
	\coordinate[label=above:11] (c) at (2,1);
	\coordinate[label=right:5] (d) at (3,0);

	\draw (a) -- (b) -- (c) -- (d) -- (a);

	\foreach \i in {a,b,c,d}
		\fill (\i) circle (2pt);
\end{tikzpicture}
\hspace{.1cm}
\begin{tikzpicture}

	\coordinate[label=below:4] (a) at (2,-1);
	\coordinate[label=left:6] (b) at (1,0);
	\coordinate[label=above:12] (c) at (2,1);
	\coordinate[label=right:8] (d) at (3,0);

	\draw (a) -- (b) -- (c) -- (d) -- (a);

	\foreach \i in {a,b,c,d}
		\fill (\i) circle (2pt);
\end{tikzpicture}
\hspace{.1cm}
\begin{tikzpicture}

	\coordinate[label=below:1] (a) at (2,-1);
	\coordinate[label=left:7] (b) at (1,0);
	\coordinate[label=above:9] (c) at (2,1);
	\coordinate[label=right:3] (d) at (3,0);

	\draw (a) -- (b) -- (c) -- (d) -- (a);

	\foreach \i in {a,b,c,d}
		\fill (\i) circle (2pt);
\end{tikzpicture}
\end{center}\caption{An example of a $3$ robot task-precedence poset whose associated permutation does not avoid $231$ and $321$.}\label{robot}
\end{figure}
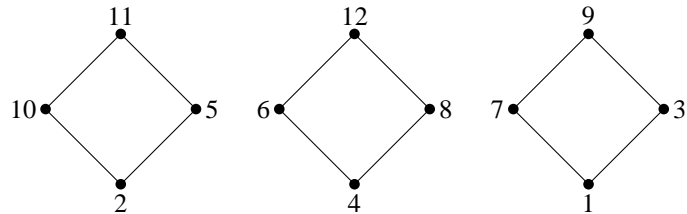

One could consider other applications that arise from task-precedence problems, but our motivating example can be generalized most appropriately by changing $4$ tasks per autonomous robot to $v$ tasks.

The $generating\phantom{.}\phantom{.}function\phantom{.}for\phantom{.}descents$ $(gfd)$ for $\mathcal{D}_{v,d}(P)$ is $f^P_{v,d}(x,y) = \sum_{D \in \mathcal{D}_{v,d}(P)} x^{des}y^d$, and $f^P_{v}(x,y) = \sum_{d=1}^{\infty}f^P_{v,d}(x,y)$. For example,  $\mathcal{D}_{4,2}(213)$ is the set of diamonds with associated permutations $1\phantom{.}2\phantom{.}3\phantom{.}4\phantom{.}5\phantom{.}6\phantom{.}7\phantom{.}8$, $1\phantom{.}2\phantom{.}3\phantom{.}8\phantom{.}4\phantom{.}5\phantom{.}6\phantom{.}7$, $1\phantom{.}2\phantom{.}7\phantom{.}8\phantom{.}3\phantom{.}4\phantom{.}5\phantom{.}6$, $1\phantom{.}6\phantom{.}7\phantom{.}8\phantom{.}2\phantom{.}3\phantom{.}4\phantom{.}5$, and $5\phantom{.}6\phantom{.}7\phantom{.}8\phantom{.}1\phantom{.}2\phantom{.}3\phantom{.}4$.
So, $f^{213}_{4,2}(x,y) = y^2(1+4x)$.

\begin{table}
\begin{center}
\begin{tabular}{|l|l|l|l|}
\hline
Patterns $P$&$\left\{\left|\mathcal{D}_{4,d}(P)\right|\right\}_{d \geq 1}$&OEIS& Result\\
\hline
$\emptyset$ & $2,280,277200,10090080000,\dots$ &A260331 & Theorem \ref{total}\\
\hline
\hline
$123$ & $0,0,0,0,0,0,0,0,0,\dots$ & A000007 & Theorem \ref{av123}\\
\hline
$132$& \multirow{2}{*}{$1,5,35,285,2530,23751,231880,2330445,\dots$} & \multirow{2}{*}{A002294} &\multirow{2}{*}{Theorem \ref{av132}}\\
$213$&&&\\
\hline
$231$& \multirow{2}{*}{$2,18,226,3298,52450,881970,\dots$} & \multirow{2}{*}{A260332} &\multirow{2}{*}{Theorem \ref{av231}}\\
$312$&&&\\
\hline
$321$ & $2,106,5976,\dots$ & A260579 & OPEN \\
\hline
\hline
$132,213$ & $1,2,4,8,16,32,64,128,256,\dots$ & A000079 & Theorem \ref{av132213}\\
\hline
$132,312$ & \multirow{2}{*}{$1,2,4,8,16,32,64,128,256,\dots$} & \multirow{2}{*}{A000079}& \multirow{2}{*}{Theorem \ref{av132312}}\\
$213,231$&&& \\
\hline
$132,321$ & \multirow{2}{*}{$1,5,13,25,41,61,85,113,145,\dots$} & \multirow{2}{*}{A001844}& \multirow{2}{*}{Theorem \ref{av132321}}\\
$213,321$&&& \\
\hline
$231,312$ & $2,8,32,128,512,2048,\dots$ & A004171 &Theorem \ref{av231312}\\
\hline
$231,321$ & \multirow{2}{*}{$2,14,98,686,4802,33614,235298,\dots$} &\multirow{2}{*}{A109808}& \multirow{2}{*}{Theorem \ref{av231321}}\\
$312,321$&&&\\
\hline
\hline
$132,213,321$&$1,2,3,4,5,6,7,8,9,\dots$&A000027&Theorem \ref{av132213321}\\
\hline
$231,312,321$ & $2, 8, 32, 128, 512, 2048, 8192, 32768, 131072\dots$ & A081294&Theorem \ref{av231312321}\\
\hline
\end{tabular}
\end{center}
\caption{Enumeration of pattern-avoiding diamonds when $v=4$}
\label{seqs}
\end{table}

Throughout this paper, the main question we answer is ``How many elements are in $\mathcal{D}_{v,d}(P)$?'' for any collection $P$ of patterns of length $3$. In general we fix $v \geq 4$ and a set of patterns $P$ and then determine a formula for the sequence $\left\{\left|\mathcal{D}_{v,d}(P)\right|\right\}_{d \geq 1}$, with key results for $v=4$ shown in Table \ref{seqs}.  The third column of the table gives entries from the Online Encyclopedia of Integer Sequences \cite{OEIS}. Our results for pattern-avoiding diamonds have connections with many other combinatorial objects, as evidenced by the low reference numbers. Sequences A260331, A260332 and A260579, however, are new results particular to this study of task precedence posets.

Our task, which answers our primary question, is to find $f^P_{v,d}(x,y)$. Then when we substitute $x = 1$ and take the coefficient of $y^d$, we obtain $\left|\mathcal{D}_{v,d}(P)\right|$.

In Section \ref{Ssingle} we consider collections of diamonds that avoid a single pattern of length 3.  In Section \ref{Spair} we consider collections of diamonds that avoid a pair of patterns of length 3, and in Section \ref{Sset} we consider collections of diamonds avoiding three or more patterns of length 3. Finally in Section \ref{open}, we list some open problems relating to this work.

\section{Diamonds avoiding a single pattern of length 3}\label{Ssingle}

Before we count pattern-avoiding diamonds, it is useful to enumerate \emph{all} diamonds.
\begin{theorem} \label{total}
$\left|\mathcal{D}_{v,d}(\emptyset)\right| = \frac{(vd)!}{v^d(v-1)^d}$
\end{theorem}

\begin{proof} Let $v \geq 4$ and $d \geq 1$, first we choose $v$ labels for each diamond, and then there are
$(v-2)!$ ways to arrange the internal vertex labels of any given diamond. We obtain
 \begin{align*}
\binom{vd}{v, \ldots, v}(v-2)!^d
&=\frac{(vd)!}{(v!)^d}\left((v-2)!\right)^d\\
&=\frac{(vd)!}{v^d(v-1)^d}.
\end{align*}
\end{proof}

\begin{theorem} \label{av123}
$\left|\mathcal{D}_{v,d}(123)\right| = 0.$
\end{theorem}

\begin{proof}
It is impossible to avoid $123$ while having a diamond since the pattern is inherent in all valid diamond labellings.
\end{proof}

\subsection{The patterns 132 and 213}

The \emph{complement} of a permutation $\pi$ of length $n$, denoted by $\pi^c$, is obtained
by replacing each letter $j$ by the letter $n$ $-$ $j+1$. The \emph{reverse} of $\pi = \pi_1 \pi_2 \ldots \pi_n$, denoted by $\pi^r$, is $\pi_n \pi_{n-1} \ldots \pi_1$. We let $\pi^{rc}$ be the \emph{reverse-complement} of $\pi$ and $\mathcal{D}_{v,d}(p)^{rc}$ be $\left\{\pi_D^{rc} \mid D \in \mathcal{D}_{v,d}(p)\right\}$.

\begin{prop}\label{Wilf}
The reverse-complement of a task precedence poset remains a legal poset and $\mathcal{D}_{v,d}(p)^{rc}=\mathcal{D}_{v,d}(p^{rc})$. In addition, for $k\geq1$, $\mathcal{D}_{v,d}(p_{1},p_{2},\ldots,p_{k})^{rc}=\mathcal{D}_{v,d}(p_{1}^{rc},p_{2}^{rc},\ldots,p_{k}^{rc}).$
\end{prop}
\begin{proof}
This is clear from the definitions and from how $\mathcal{D}_{v,d}(p)^{rc}$ is created from $p$.
\end{proof}

Thus we immediately see that a) $132 \sim_W 213$, b) $231 \sim_W 312$, c) $132, 312 \sim_W 213, 231$, d) $132, 321 \sim_W 213, 321$, and e) $231, 321 \sim_W 312, 321$.

Given a permutation $\pi$ in $S_n$, $lis(\pi)$ is the length of a longest increasing subsequence in $\pi$. For example, in the permutation $1\phantom{.}2\phantom{.}5\phantom{.}6\phantom{.}3\phantom{.}4\phantom{.}7\phantom{.}8$ a longest increasing subsequence is $1\phantom{.}2\phantom{.}5\phantom{.}6\phantom{.}7\phantom{.}8$ and $lis($1\phantom{.}2\phantom{.}5\phantom{.}6\phantom{.}3\phantom{.}4\phantom{.}7\phantom{.}8$) = 6$. Given a permutation $\pi$ in $S_n$, $rlmax(\pi)$ is the number of right-left maxima in $\pi$.
 For example, in the permutation $2\phantom{.}4\phantom{.}6\phantom{.}8\phantom{.}1\phantom{.}3\phantom{.}5\phantom{.}7$ a maximum is reached when reading right-to-left twice and $rlmax($2\phantom{.}4\phantom{.}6\phantom{.}8\phantom{.}1\phantom{.}3\phantom{.}5\phantom{.}7$) = 2$.
Let $Dyck_{v,d}$ be the set of all paths from $(0,0)$ to $(d,vd)$ using only $(0,1)$ and $(1,0)$ steps (East and North steps) which stay weakly under $y = vx$. Given any $p$ $\in$ $Dyck_{v,d}$, $touchpoints(p)$ is the number of times $p$ touches the line $y = vx$, excluding the point $(v,vd)$.
In Figure 2.1, the Dyck path touches the line $y = 4x$ three times and $touchpoints(p) = 3$.
Given any $p$ $\in$ $Dyck_{v,d}$, $corners(p)$  is the number of North steps that are followed by one or more East steps in $p$.
In Figure 2.1, there are three places where the Dyck path has one or more North steps followed by one or more East steps and $corners(p) = 3$.
Given any $p$ $\in$ $Dyck_{v,d}$, $height(p)$ is the greatest vertical distance from any point on $p$ to the line $y = vx$.
 In figure 2.1, the longest distance from a corner in the Dyck path to the line $y = 4x$ is seven (from $(3,5)$ to $(3,12)$) and $height(p) = 7$.

\begin{lemma}\label{lem132} Any element of $\mathcal{D}_{v,d}(132)$ has the elements on each diamond labelled in increasing order. \end{lemma}
Otherwise the label of the first element of the diamond together with the first descent would form a $132$ pattern.

\begin{theorem} \label{av132}
\[\sum_{\sigma\in\mathcal{D}_{v,d(132)}}w^{rlmax(\sigma)}x^{des(\sigma)}y^{d}z^{lis(\sigma)} = \sum_{p\in Dyck_{v,d}}w^{touchpoints(p)}x^{corners(p)}y^{d}z^{height(p)}.\]
\end{theorem}
\begin{proof}

We define a map $\phi$ from $Dyck_{v,d}$ to $\mathcal{D}_{v,d}(132)$. To find $\phi(p)$, first write out the heights of the East steps. For each height, include a subscript $j$ that indicates how many East steps are at that height. Reverse this sequence and add $1$ to every item in the list, leaving the subscripts unchanged.  Each of the elements of this list becomes the first label of a diamond, and then place $vj$ labels in increasing order using the smallest elements that have not already been used as labels.

As an example, refer to Figure 2.1. The heights of the East steps are $0,4,5,12$. When this sequence is reversed and $1$ is added to each term, the resulting sequence is $13_1,6_1,5_1,1_1$. Thus the permutation associated with this Dyck path is $13\phantom{.}14\phantom{.}15\phantom{.}16\phantom{.}6\phantom{.}7\phantom{.}8\phantom{.}9\phantom{.}5\phantom{.}
10\phantom{.}11\phantom{.}12\phantom{.}1\phantom{.}2\phantom{.}3\phantom{.}4$.

The importance of the subscripts $j$ are evident from the image of Figure 2.3 under $\phi$. The heights of the East steps are $0, 3, 3, 10$, and the resulting sequence is $11_1, 4_2, 1_1$. Thus the permutation associated to image is $11\phantom{.}12\phantom{.}13\phantom{.}14\phantom{.}4\phantom{.}5\phantom{.}6\phantom{.}7\phantom{.}
8\phantom{.}9\phantom{.}10\phantom{.}15\phantom{.}1\phantom{.}2\phantom{.}3\phantom{.}16.$

This map is certainly reversible, with the first label on each diamond forming a list, unless there is an increase between diamonds, in which case the first label is repeated. Then the list is reversed and $1$ is subtracted from each element, giving us the heights of the East steps in the Dyck path.

This bijection is particularly natural when you examine common statistics on both paths and permutations. Following touchpoints, corners, and height through the bijection, we find they correspond exactly to right-left maximum, descents, and longest increasing sequence on the permutation.

\begin{figure}
\begin{center}
\begin{tikzpicture}

	\coordinate[label=left:{(0,0)}] (a) at (0,0);
	\coordinate[label=left:{(4,16)}] (b) at (4,4);
	\coordinate (c) at (1,0);
	\coordinate[label=left:{(1,4)}] (d) at (1,1);
	\coordinate (e) at (2,1);
	\coordinate[label=above:{(2,5)}] (f) at (2,5/4);
	\coordinate (g) at (3,5/4);
	\coordinate[label=left:{(3,12)}] (h) at (3,3);
	\coordinate (i) at (4,3);
	
	\draw (a) -- (b);
	\draw (a) -- (c) -- (d) -- (e) -- (f) -- (g) -- (h) -- (i) -- (b);

\end{tikzpicture}
\end{center}
\caption{A Dyck path from (0,0) to (4,16)\label{dyck1}}
\end{figure}
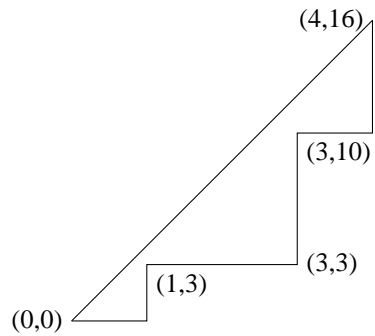

\begin{figure}
\begin{center}
\begin{tikzpicture}

	\coordinate[label=below:13] (a) at (2,-1);
	\coordinate[label=left:14] (b) at (1,0);
	\coordinate[label=above:16] (c) at (2,1);
	\coordinate[label=right:15] (d) at (3,0);

	\draw (a) -- (b) -- (c) -- (d) -- (a);

	\foreach \i in {a,b,c,d}
		\fill (\i) circle (2pt);
\end{tikzpicture}
\hspace{.3cm}
\begin{tikzpicture}

	\coordinate[label=below:6] (a) at (2,-1);
	\coordinate[label=left:7] (b) at (1,0);
	\coordinate[label=above:9] (c) at (2,1);
	\coordinate[label=right:8] (d) at (3,0);

	\draw (a) -- (b) -- (c) -- (d) -- (a);

	\foreach \i in {a,b,c,d}
		\fill (\i) circle (2pt);
\end{tikzpicture}
\hspace{.3cm}
\begin{tikzpicture}

	\coordinate[label=below:5] (a) at (2,-1);
	\coordinate[label=left:10] (b) at (1,0);
	\coordinate[label=above:12] (c) at (2,1);
    \coordinate[label=right:11] (d) at (3,0);

	\draw (a) -- (b) -- (c) -- (d) -- (a);

	\foreach \i in {a,b,c,d}
		\fill (\i) circle (2pt);

\end{tikzpicture}
\hspace{.3cm}
\begin{tikzpicture}

	\coordinate[label=below:1] (a) at (2,-1);
	\coordinate[label=left:2] (b) at (1,0);
	\coordinate[label=above:4] (c) at (2,1);
    \coordinate[label=right:3] (d) at (3,0);

	\draw (a) -- (b) -- (c) -- (d) -- (a);

	\foreach \i in {a,b,c,d}
		\fill (\i) circle (2pt);
\end{tikzpicture}
\end{center}
\caption{Diamonds labelled according to the image of Figure \ref{dyck1} under the bijection}
\end{figure}
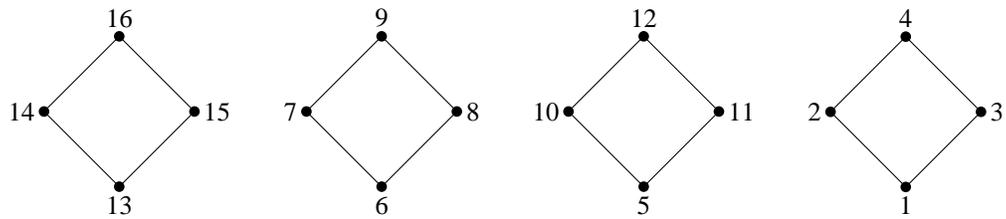

\begin{figure}
\begin{center}
\begin{tikzpicture}

	\coordinate[label=left:{(0,0)}] (a) at (0,0);
	\coordinate[label=left:{(4,16)}] (b) at (4,4);
	\coordinate (c) at (1,0);
	\coordinate[label=below right:{(1,3)}] (d) at (1,3/4);
	\coordinate[label=right:{(3,3)}] (e) at (3,3/4);
	\coordinate[label=below right:{(3,10)}] (f) at (3,10/4);
	\coordinate (g) at (4,10/4);
	
	\draw (a) -- (b);
	\draw (a) -- (c) -- (d) -- (e) -- (f) -- (g) -- (b);

\end{tikzpicture}
\end{center}
\caption{A second Dyck path from $(0,0)$ to $(4,16)$.}
\end{figure}
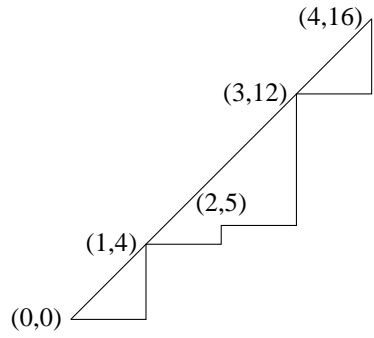

\end{proof}

\begin{corollary}
$\left|\mathcal{D}_{v,d}(132)\right| =\left|\mathcal{D}_{v,d}(213)\right| = |\text{Dyck}_{v,d}| = \frac{\binom{d(v+1)}{d}}{(vd+1)}. $\cite{Fuss1791}
\end{corollary}

\begin{proof}
These equalities hold by the bijection in Theorem \ref{av132} and trivial Wilf equivalence from Proposition \ref{Wilf}.
\end{proof}

\subsection{The patterns 231 and 312}

 Consider $D$ in $\mathcal{D}_{v,d}(231)$, and suppose label $vd$ occurs in position $k$. Then for all $i < k$ and for all $j>k$, $a_i < a_j$. Consequently, if label $vd$ is in position $k$, then labels $(1,\ldots,k-1)$ appear in positions $(1,\ldots,k-1)$. We define $\mathcal{D}_{v,j}^d$ to be the collection of labelled diamonds for $d-1$ full diamonds with $v$ vertices each followed by an incomplete diamond with $j$ vertices for $j = 1,\ldots, v-1$. Likewise $\mathcal{D}_{v,j}^d(p)$ are those diamonds that avoid pattern $p$. Note, when $j=1$ there exist no order relations in the final partial diamond. An example is shown in Figure \ref{tikz2}.

\begin{figure}
\begin{center}
\begin{tikzpicture}

	\coordinate (a) at (2,-1);
	\coordinate (b) at (1,0);
	\coordinate (c) at (2,1);
	\coordinate (d) at (3,0);
    \coordinate (e) at (1.5,0);
    \coordinate (f) at (1.9,0);
    \coordinate (g) at (2.25,0);
    \coordinate (h) at (2.6,0);

	\draw (a) -- (b) -- (c) -- (d) -- (a);
    \draw (a) -- (e) -- (c);

	\foreach \i in {a,b,c,d,e}
		\fill (\i) circle (2pt);
    \foreach \i in {f,g,h}
        \fill (\i) circle (.75pt);
\end{tikzpicture}
\hspace{.3cm}
\begin{tikzpicture}

    \coordinate (a) at (1,-1.065);
    \coordinate (b) at (1,0);
	\coordinate (c) at (1.35,0);
	\coordinate (d) at (1.7,0);

    \foreach \i in {a}
        \fill (\i) circle (0pt);
    \foreach \i in {b,c,d}
        \fill (\i) circle (.75pt);

\end{tikzpicture}
\hspace{.3cm}
\begin{tikzpicture}

	\coordinate (a) at (2,-1);
	\coordinate (b) at (1,0);
	\coordinate (c) at (2,1);
	\coordinate (d) at (3,0);
    \coordinate (e) at (1.5,0);
    \coordinate (f) at (1.9,0);
    \coordinate (g) at (2.25,0);
    \coordinate (h) at (2.6,0);

	\draw (a) -- (b) -- (c) -- (d) -- (a);
    \draw (a) -- (e) -- (c);

	\foreach \i in {a,b,c,d,e}
		\fill (\i) circle (2pt);
    \foreach \i in {f,g,h}
        \fill (\i) circle (.75pt);
\end{tikzpicture}
\hspace{.3cm}
\begin{tikzpicture}

	\coordinate (a) at (2,-1);
	\coordinate (b) at (1,0);
	\coordinate (d) at (3,0);
    \coordinate (e) at (1.5,0);
    \coordinate (f) at (1.9,0);
    \coordinate (g) at (2.25,0);
    \coordinate (h) at (2.6,0);

	\draw (a) -- (b);
    \draw (a) -- (d);
    \draw (a) -- (e);

	\foreach \i in {a,b,d,e}
		\fill (\i) circle (2pt);
    \foreach \i in {f,g,h}
        \fill (\i) circle (.75pt);
\end{tikzpicture}
\end{center}\caption{An unlabelled member of $\mathcal{D}_{v,j}^d$ for $d /geq 3$, $v /geq 5$, and $j/geq 4$.}\label{tikz2}
\end{figure}
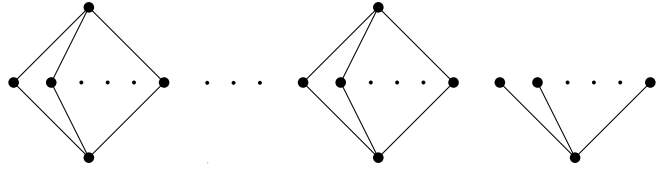

$\alpha_{v,j}^{d}(x)$ (or sometimes simply $\alpha_{v,j}^d$ for brevity) is the generating function for descents in $\mathcal{D}_{v,j}^d(231)$. In other words, \[\alpha_{v,j}^{d}(x) = \sum_{D\in\mathcal{D}_{v,j}^d(231)}x^{des(\pi_D)}.\]


For example, $\mathcal{D}_{5,1}^{2}(231)$ contains the diamonds with the following associated permutations: $123456,$ $124356,$ $142356,$ $132456,$ $143256,$ $123465,$ $124365,$ $142365,$ $132465,$ $143265.$ Counting descents in these ten permutations gives the generating function for descents $\alpha_{5,1}^2(x) = 1+4x+4x^2+x^3.$
\begin{theorem} \label{av231}
\[{f}_{v,d}^{(231)}(x,1)=\alpha_{v,v}^d(x)=\alpha_{v,(v-1)}^{d}+x\displaystyle \sum\limits_{i=1}^{d-1}\alpha_{v,(v-1)}^{i}\alpha_{v,v}^{d-i}\] where \[\alpha_{v,j}^{1}=
\begin{cases}
    1,& \text{if } j=1\\
    C_{j-1},& \text{if } j=2,\ldots,v-1\\
    C_{v-2},& \text{if } j=v
\end{cases}.\]
and $C_i$ is the $i^{th}$ Catalan number.
\end{theorem}

\begin{proof}
We proceed by partitioning elements of $\mathcal{D}_{v,j}^d(231)$ by where the largest label occurs. Let $m=v(d-1)+j$ be the largest label in $(d-1)$ diamonds with $v$ vertices followed by an incomplete diamond with $j$ vertices.

Now, assume $j=1.$ The $m$ label can appear on the final element or on the greatest element of any of the full diamonds. When $m$ occurs on the final least element there are $(d-1)$ diamonds with $v$ vertices that precede $m,$ so we then have $\alpha_{v,v}^{d-1}$ as the generating function for descents (gfd) for the vertices before $m$ that will avoid $231.$ When $m$ appears on the greatest element of the $i^{th}$ complete diamond, $1/leq i/leq d-1$, we have $\alpha_{v, v-1}^i$ as the gfd for the vertices before $m$, and $\alpha_{v,1}^{d-i}$ as the gfd for the vertices following $m$. Because we have created a descent from $m$ to the least element of the next diamond or partial diamond, we must also multiply by $x$ to account for this extra descent.

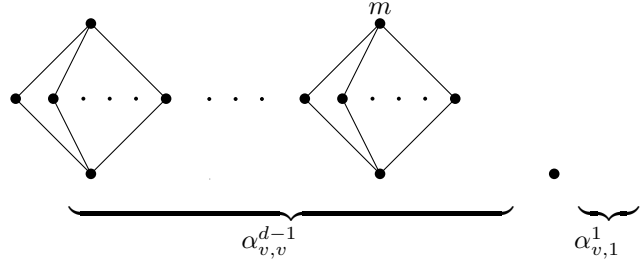
\begin{figure}
\begin{center}
\begin{tikzpicture}

	\coordinate (a) at (2,-1);
	\coordinate (b) at (1,0);
	\coordinate (c) at (2,1);
	\coordinate (d) at (3,0);
    \coordinate (e) at (1.5,0);
    \coordinate (f) at (1.9,0);
    \coordinate (g) at (2.25,0);
    \coordinate (h) at (2.6,0);

	\draw (a) -- (b) -- (c) -- (d) -- (a);
    \draw (a) -- (e) -- (c);

	\foreach \i in {a,b,c,d,e}
		\fill (\i) circle (2pt);
    \foreach \i in {f,g,h}
        \fill (\i) circle (.75pt);
\end{tikzpicture}
\hspace{.3cm}
\begin{tikzpicture}

    \coordinate (a) at (1,-1.065);
    \coordinate (b) at (1,0);
	\coordinate (c) at (1.35,0);
	\coordinate (d) at (1.7,0);

    \foreach \i in {a}
        \fill (\i) circle (0pt);
    \foreach \i in {b,c,d}
        \fill (\i) circle (.75pt);

\end{tikzpicture}
\hspace{.3cm}
\begin{tikzpicture}

	\coordinate (a) at (2,-1);
	\coordinate (b) at (1,0);
	\coordinate[label=above: $m$] (c) at (2,1);
	\coordinate (d) at (3,0);
    \coordinate (e) at (1.5,0);
    \coordinate (f) at (1.9,0);
    \coordinate (g) at (2.25,0);
    \coordinate (h) at (2.6,0);

	\draw (a) -- (b) -- (c) -- (d) -- (a);
    \draw (a) -- (e) -- (c);

	\foreach \i in {a,b,c,d,e}
		\fill (\i) circle (2pt);
    \foreach \i in {f,g,h}
        \fill (\i) circle (.75pt);
\end{tikzpicture}
\hspace{1cm}
\begin{tikzpicture}

	\coordinate (a) at (2,-1);

	\foreach \i in {a}
		\fill (\i) circle (2pt);

\end{tikzpicture}
\end{center}
\vspace{-0.3cm}
\hspace{4.6cm}
$\underbrace{\hspace{5.9cm}}$
\hspace{0.7cm}
$\underbrace{\hspace{0.8cm}}$\\
\vspace{-0.3cm}
\hspace{6.8cm} $\alpha_{v,v}^{d-1}$
\hspace{3.5cm} $\alpha_{v,1}^1$\\
\caption{$\alpha_{v,1}^d$ when $m$ appears on the greatest element of the last full diamond.}
\end{figure}

Hence \[\alpha_{v,1}^{d}(x)=\alpha_{v,v}^{d-1}+x\displaystyle \sum\limits_{i=1}^{d-1}\alpha_{v,(v-1)}^{i}\alpha_{v,1}^{d-i}.\]

Now, assume we have $(d-1)$ diamonds followed by an incomplete diamond with $j$ vertices where $j=2,\ldots,v-1.$ The $m^{th}$ element can appear on any of the interior vertices but not on the least element of the incomplete diamond, or $m$ can appear on the greatest element of any complete diamond. When $m$ appears on any of the interior vertices of the final diamond we need to count the descents before $m,$ after $m,$ and from $m$ itself. The descents that occur before $m$ can be counted by $\alpha_{v,j-g}^{d}$ where $g$ is the number of interior vertices following $m$ including $m.$ The descents following $m$ are counted by $\alpha_{v,g}^{1}$ because the same number of descents can occur in the remaining interior vertices as when we have a single incomplete diamond. We then count the descent that results from $m$ by multiplying our gfd by $x,$ but we do not get a descent from $m$ when it appears on the final interior vertex. We then sum over all possible values of $g$ to give us the gfd when $m$ appears on the interior vertices of the final diamond which gives us $\alpha_{v,j-1}^{d}\alpha_{v,1}^{1}+x\displaystyle \sum\limits_{g=2}^{j-1}\alpha_{v,j-g}^{d}\alpha_{v,g}^{1}.$

\begin{figure}
\begin{center}
\begin{tikzpicture}

	\coordinate (a) at (2,-1);
	\coordinate (b) at (1,0);
	\coordinate (c) at (2,1);
	\coordinate (d) at (3,0);
    \coordinate (e) at (1.5,0);
    \coordinate (f) at (1.9,0);
    \coordinate (g) at (2.25,0);
    \coordinate (h) at (2.6,0);

	\draw (a) -- (b) -- (c) -- (d) -- (a);
    \draw (a) -- (e) -- (c);

	\foreach \i in {a,b,c,d,e}
		\fill (\i) circle (2pt);
    \foreach \i in {f,g,h}
        \fill (\i) circle (.75pt);
\end{tikzpicture}
\hspace{.3cm}
\begin{tikzpicture}

    \coordinate (a) at (1,-1.065);
    \coordinate (b) at (1,0);
	\coordinate (c) at (1.35,0);
	\coordinate (d) at (1.7,0);

    \foreach \i in {a}
        \fill (\i) circle (0pt);
    \foreach \i in {b,c,d}
        \fill (\i) circle (.75pt);

\end{tikzpicture}
\hspace{.3cm}
\begin{tikzpicture}

	\coordinate (a) at (2,-1);
	\coordinate (b) at (1,0);
	\coordinate (c) at (2,1);
	\coordinate (d) at (3,0);
    \coordinate (e) at (1.5,0);
    \coordinate (f) at (1.9,0);
    \coordinate (g) at (2.25,0);
    \coordinate (h) at (2.6,0);

	\draw (a) -- (b) -- (c) -- (d) -- (a);
    \draw (a) -- (e) -- (c);

	\foreach \i in {a,b,c,d,e}
		\fill (\i) circle (2pt);
    \foreach \i in {f,g,h}
        \fill (\i) circle (.75pt);
\end{tikzpicture}
\hspace{.3cm}
\begin{tikzpicture}

	\coordinate (a) at (2,-1);
	\coordinate (b) at (1,0);
	\coordinate (d) at (3,0);
    \coordinate[label=above: $m$] (e) at (1.5,0);
    \coordinate (f) at (1.9,0);
    \coordinate (g) at (2.25,0);
    \coordinate (h) at (2.6,0);

	\draw (a) -- (b);
    \draw (a) -- (d);
    \draw (a) -- (e);

	\foreach \i in {a,b,d,e}
		\fill (\i) circle (2pt);
    \foreach \i in {f,g,h}
        \fill (\i) circle (.75pt);
\end{tikzpicture}
\end{center}

\vspace{-0.3cm}
\hspace{3.75cm}
$\underbrace{\hspace{6.9cm}}$
\hspace{0.15cm}
$\underbrace{\hspace{1.5cm}}$\\
\vspace{-0.3cm}
\hspace{6.5cm} $\alpha_{v,(j-g)}^{d}$
\hspace{3.2cm} $\alpha_{v,g}^1$\\
\caption{$\alpha_{v,j}^d$ when $m$ appears on the final (partial) diamond.}
\end{figure}
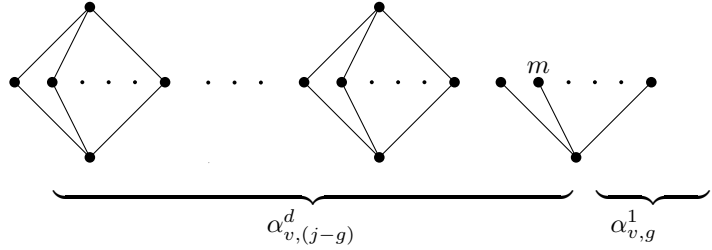

Also, $m$ can appear on the greatest element of any of the full diamonds. When $m$ appears on the greatest element of the $i^{th}$ complete diamond the gfd for vertices that appear before $m$ is $\alpha_{v,(v-1)}^{i}$ and $\alpha_{v,j}^{d-i}$ for the vertices following $m.$ We count the descent from $m$ by multiplying our gfd by $x.$ The total gfd when $m$ appears on the greatest element of the $i^{th}$ diamond is then $\alpha_{v,(v-1)}^{i}\alpha_{v,j}^{d-i}x.$

Thus \[\alpha_{v,j}^d(x)=\alpha_{v,j-1}^{d}\alpha_{v,1}^{1}+x\displaystyle \sum\limits_{g=2}^{j-1}\alpha_{v,j-g}^{d}\alpha_{v,g}^{1}+x\displaystyle \sum\limits_{i=1}^{d-1}\alpha_{v,(v-1)}^{i}\alpha_{v,j}^{d-i}\]
for $j=2,\ldots,v-1.$

Lastly, we look at when we have $d$ complete diamonds. The $m^{th}$ element can appear on any of the greatest elements. When $m$ appears on the greatest element of the last diamond, the gfd is $\alpha_{v,(v-1)}^{d}$ which counts descents before $m.$

When $m$ appears on the greatest element of the $i^{th}$ complete diamond ($1 \leq  i \leq d-1$), the gfd for vertices that appear before $m$ is $\alpha_{v,(v-1)}^{i}$ and $\alpha_{v,v}^{d-i}$ for vertices following $m.$ We count the descent from $m$ to the following least element by multiplying the gfd by $x.$ 

Hence \[\alpha_{v,v}^d(x)={f}_{v,d}^{(231)}(x,1)=\alpha_{v,(v-1)}^{d}+x\displaystyle \sum\limits_{i=1}^{d-1}\alpha_{v,(v-1)}^{i}\alpha_{v,v}^{d-i}.\]
\end{proof}

We can use this result to recursively obtain $f_{v,d}^{231}(x,1)$ for any $v$ and $d$.
\begin{corollary}
$f^{231}_{v,d}(1,y)\Bigr|_{y^d} = \alpha_{v,v}^{d}(1) = \left|\mathcal{D}_{v,d}(231)\right|$
\end{corollary}
Tables \ref{table231} and \ref{table231b} are an example of the steps of such a computation for $\alpha_{5,5}^3(x)$ and $\mathcal{D}_{5,3}(231).$

\begin{table}
\begin{center}
\begin{tabular}{ |p{.65cm}||p{1.7cm}|p{3.8cm}|p{4.2cm}|  }
 \hline
 \multicolumn{4}{|c|}{v=5} \\
 \hline
 d & 1 & 2 & 3 \\
 \hline
 $\alpha_{5,1}^{d}$ & $1$        & $1+4x+4x^2+x^3$                & $1+13x+54x^2+95x^3+74x^4+25x^5+3x^6$ \\
 \hline
 $\alpha_{5,2}^{d}$ & $1$        & $1+5x+7x^2+2x^3$               & $1+15x+72x^2+149x^3+138x^4+53x^5+7x^6$ \\
 \hline
 $\alpha_{5,3}^{d}$ & $1+x$      & $1+7x+15x^2+10x^3+2x^4$        & $1+18x+106x^2+281x^3+362x^4+225x^5+65x^6+7x^7$ \\
 \hline
 $\alpha_{5,4}^{d}$ & $1+3x+x^2$ & $1+10x+31x^2+36x^3+15x^4+2x^5$ & $1+22x+161x^2+544x^3+938x^4+840x^5+383x^6+84x^7+7x^8$ \\
 \hline\hline
 $\alpha_{5,5}^{d}$ & $1+3x+x^2$ & $1+11x+37x^2+47x^3+21x^4+3x^5$ & $1+24x+188x^2+677x^3+1246x^4+1193x^5+579x^6+135x^7+12x^8$ \\
 \hline
\end{tabular}
\end{center}\caption{The recursive steps necessary to find the generating function for descents in $\mathcal{D}_{5,3}(231).$}\label{table231}
\end{table}

\begin{table}
\begin{center}
\begin{tabular}{ |p{1cm}||p{2cm}|p{2cm}|p{2cm}|  }
 \hline
 \multicolumn{4}{|c|}{v=5} \\
 \hline
 d & 1 & 2 & 3 \\
 \hline
 $\alpha_{5,1}^{d}$ & $1$ & $10$  & $265$ \\
 \hline
 $\alpha_{5,2}^{d}$ & $1$ & $15$  & $435$ \\
 \hline
 $\alpha_{5,3}^{d}$ & $2$ & $35$  & $1065$ \\
 \hline
 $\alpha_{5,4}^{d}$ & $5$ & $95$  & $2980$ \\
 \hline\hline
 $\alpha_{5,5}^{d}$ & $5$ & $120$ & $4055$ \\
 \hline
\end{tabular}
\end{center}\caption{The total number of permutations for $\mathcal{D}_{5,d}$ that avoid the pattern $231$ for $d=1,2,3$.}\label{table231b}
\end{table}

\begin{corollary} \label{av312}
$\left|\mathcal{D}_{v,d}(231)\right| = \left|\mathcal{D}_{v,d}(312)\right|.$
\end{corollary}

\begin{proof}
By Proposition \ref{Wilf}, $231$ is $d$-Wilf-equivalent to $312.$
\end{proof}

\subsection{The pattern 321}

We were unable to find a closed formula for the pattern $321$. In Table \ref{table321}, we present the first few terms of the sequence and the first few generating functions for descents, which we found using Sage. We are confident that a technique recently used by Bevan, et.al. \cite{Shrubs} would be successful in this case too. Their technique involved refining a bivariate generating function via a statistic called last inversion foot, using a result of Bousquet-M\'{e}lou, and finding a functional equation, to eventually give a growth rate for the sequence. This suspicion was confirmed by Bevan, and in fact the sequence begins: 2, 106, 5976, 387564, 27247446, 2020632046, 155622020610, 12327937844924, 998103225615208, 82224228576059340 \cite{Bevan}. However the authors were unable, in the time available for this project, to learn all the tools necessary to enact the technique and so the problem remains officially open.

\begin{table}
\begin{center}
\begin{tabular}{ |p{1in}|p{1in}|p{3in}| }
\hline
 \multicolumn{3}{|c|}{v=4} \\
 \hline
\hline
$d$ & $|{\mathcal{D}_{v,d}(321)}|$ & $f_{v,d}^{321}(x)$ \\
\hline
 1 & 2 & $1+x$ \\ \hline
 2 & 106 & $1+71x+29x^2+5x^3$ \\ \hline
 3 & 5976 & $1+991x+2747x^2+1765x^3+430x^4+42x^5$ \\

\hline
\end{tabular}
\end{center}\caption{$|{\mathcal{D}_{v,d}(321)}|$ and $f_{v,d}^{321}(x)$ for $d = 1, 2, 3$.}\label{table321}
\end{table}

\section{Diamonds avoiding a pair of patterns of length 3}\label{Spair}

Next, we study pairs of patterns of length $3$.  While there are $15$ such pairs of patterns, we focus on the $8$ pairs of patterns $\sigma, \rho$ where $\left|\mathcal{D}_{v,d}(\sigma,\rho)\right|$ is non-trivial.

\subsection{Diamonds avoiding the set of patterns $132, 213$}

\begin{lemma}\label{lem132213} In order to avoid $132$ and $213$, the labels on each diamond must be increasing and consecutive.\end{lemma}

\begin{proof} By Lemma \ref{lem132}, the labels appear in increasing order on each diamond. Then any label ``missing" from consecutive labelling would either create  $213$ if it occurred before its surrounding labels, or a $132$ if it occurred after. Therefore the labels on each diamond must be consecutive and increasing. \end{proof}

\begin{theorem} \label{av132213}
$f^{132,213}_{v,d}(x,y)= \displaystyle \sum\limits_{d=1}^\infty \sum\limits_{\sigma \in \mathcal{D}(132,213)} y^dx^{des(\sigma)} = \frac{1-yx}{1-y(1+x)}.$
\end{theorem}

\begin{proof}
By Lemma \ref{lem132213}, we know that the labels on each diamond are consecutive and increasing, so there is a diamond labelled $1$, $2$, $\ldots$, $v$, another labelled $v+1$, $\ldots$, $2v$, etc. So the only thing we must ensure is that the entire collection of diamonds avoids $132$ and $213$ between the respective diamonds. In their foundational paper, Simion and Schmidt enumerated permutations avoiding $132$ and $213$ \cite{SimSchmidt}, and the recursive nature of their proof can also be adapted to find our generating function for descents.

The labels $v(d-1) + 1, \ldots, vd$ must occur on either the first diamond, or the last. In the first case, they create a descent. In the second, they do not, giving a $(1+x)$ term in the generating function. We continue recursively and obtain:

\begin{center}
\begin{align*}
\sum_{D \in \mathcal{D}_{v,d}(132,213)} y^dx^{des(\sigma(D))} &= 1 + \sum\limits_{d=1}^\infty y^d(1+x)^{d-1}\\
&= 1 + \frac{1}{1+x}\left( -1 + \frac{1}{1-y(1+x)}\right)\\
&= \frac{x}{1+x} + \frac{1}{(1+x)(1-y(1+x))}\\
&= \frac{x(1-y(1+x))+1}{(1+x)(1-y(1+x))}\\
&= \frac{1-yx}{1-y(1+x).}
\end{align*}
\end{center}

\end{proof}

\begin{corollary}
$f^{132,213}_{v,d}(1,y)\Bigr|_{y^d} = \left|\mathcal{D}_{v,d}(132, 213)\right| = 2^{d-1}.$
\end{corollary}

\subsection{Diamonds avoiding $132, 312$ and $213,231$}

\begin{lemma}\label{lem132312}
In order to avoid $132$ and $312$, the final diamond is labelled with either $v(d-1)+1, v(d-1)+2, \ldots, vd$ or $1, v(d-1)+2, \ldots, vd$
\end{lemma}
\begin{proof} Since $v\geq 4$, the label $vd$ must appear on the final diamond in order to avoid $312$. Likewise the interior vertices on the final diamond must be in consecutive increasing order in order to avoid $132$, so the $v-1$ final vertices are $v(d-1) + 2, \ldots, vd$. If the label on the first vertex of the last diamond were some number $j$ other than $1$ or $v(d-1) +1$, then the first vertex of whichever diamond $v(d-1)+1$, along with $v(d-1)+1$, and $j$ would form a $132$.
\end{proof}

\begin{theorem} \label{av132312}
$f^{132,312}_{v,d}(x,y)= \displaystyle \sum\limits_{d=1}^\infty \sum\limits_{\sigma \in \mathcal{D}(132,312)} y^dx^{des(\sigma)} = \frac{1-yx}{1-y(1+x)}.$
\end{theorem}

\begin{proof}
 We proceed similarly to the proof of Theorem \ref{av132213} with a recursive argument. By Lemma \ref{lem132213}, the final diamond has only two possibilities, one of which forms a descent with the previous diamond, and one of which doesn't. Thus our descent generating function gains a $(1+x)$ term for each additional diamond, and exactly as in Theorem \ref{av132213}, the result follows.
\end{proof}

\begin{corollary}\label{av213231}
$f^{132,312}_{v,d}(1,y)\Bigr|_{y^d} = \left|\mathcal{D}_{v,d}(132, 312)\right| = \left|\mathcal{D}_{v,d}(213, 231)\right| = 2^{d-1}.$
\end{corollary}


\begin{proof}
By Proposition \ref{Wilf}, $213,231$ is $d$-Wilf-equivalent to $132,312.$
\end{proof}

\subsection{Diamonds avoiding $132, 321$ or $213, 321$}

\begin{lemma}\label{lem132321} Any diamond avoiding $132$ and $321$ has at most one descent. Moreover, if there is a descent, it involves the label $1$.\end{lemma}
\begin{proof} By examination of cases, any arrangement of two descents forms either a $132$ or a $321$. If a descent does not involve the $1$, then either the $1$ occurs before, causing a $132$, or the $1$ occurs after, causing a $321$. \end{proof}

\begin{theorem} \label{av132321}
$f^{132,321}_{v,d}(x,y)= \displaystyle \sum\limits_{d=1}^\infty \sum\limits_{\sigma \in \mathcal{D}(132,321)} y^dx^{des(\sigma)} = \frac{1-2y+y^2+vxy^2}{(1-y)^3}.$
\end{theorem}

\begin{proof}

By Lemma \ref{lem132321}, we need only enumerate those diamonds with one descent where the descent involves the $1$. Everything after the $1$ increases, as does everything before the $1$. In fact, the permutations associated to diamonds that avoid $132$ and $321$ look like a portion of the identity permutation was deleted from the front and inserted after position $vi$, for $i=1, \ldots, d-1$. When $i=d-1$, there are $v$ possibilities for how many numbers appear consecutively with $1$, including $1$. When $i = d-2$, there are $2v$ possibilities, etc. When $i = 1$, there are $(d-1)v$ possibilities. Thus we have $\frac{v}{2}d(d-1)$ diamonds with one descent, and one with zero descents. Thus,

\begin{center}
\begin{align*}
\sum\nolimits y^dx^{des(\sigma)} &= \sum\limits_{d=0}^\infty y^d[1+\frac{v}{2}d(d-1)x]\\
&= \sum\limits_{d=0}^\infty y^d + \frac{vx}{2}\sum\limits_{d=0}^\infty y^dd(d-1)\\
&= \frac{1}{1-y} + \frac{vx}{2}(y^2)\left(\frac{2}{(1-y)^3}\right)\\
&= \frac{1}{1-y} + \frac{vxy^2}{(1-y)^3}\\
&= \frac{(1-y)^2}{(1-y)^3} + \frac{vxy^2}{(1-y)^3}\\
&= \frac{1-2y+y^2+vxy^2}{(1-y)^2}.
\end{align*}
\end{center}

\end{proof}

\begin{corollary}\label{av213321}
$f^{132,321}_{v,d}(1,y)\Bigr|_{y^d} = \left|\mathcal{D}_{v,d}(132, 321)\right| = \left|\mathcal{D}_{v,d}(213, 321)\right| = 1+v\left(\frac{d(d-1)}{2}\right)$
\end{corollary}


\begin{proof}
$\mathcal{D}_{v,d}(213,321)$ and $\mathcal{D}_{v,d}(132, 321)$ are $d$-Wilf Equivalent by Proposition \ref{Wilf}.
\end{proof}

\subsection{Diamonds avoiding $231, 312$}

\begin{theorem} \label{av231312}
$f^{231,312}_{v,d}(x,y)= \displaystyle \sum\limits_{d=1}^\infty \sum\limits_{\sigma \in \mathcal{D}(231,312)} y^dx^{des(\sigma)} = \frac{x+yx(1+x)^{v-2}+1}{(1+x)(1-y(1+x)^{v-2})}.$
\end{theorem}

\begin{proof}
Let $n = vd$ be the largest label on a diamond $D \in \mathcal{D}_{v,d}(231,312)$. Avoiding the pattern $231$ means the $1$ must be at the beginning and avoiding the pattern $312$ implies everything after $n$ must be decreasing which forces $n$ to the end of the permutation.
By a result of Simion and Schmidt on permutations, there are $2^{v-3}$ ways to arrange the middle-level vertices within each of the $d$ diamonds in order to avoid both $231$ and $312$ creating between $0$ and $v-3$ descents \cite{SimSchmidt}. There are also two ways to either swap or not swap the last element of each diamond with the first element of the next. This gives the following generating function.

\begin{center}
\begin{align*}
\sum_{D \in \mathcal{D}_{v,d}(231,312)} y^dx^{des(\sigma(D))} &= 1 + \sum\limits_{d=1}^\infty y^d(1+x)^{(v-2)d-1}\\
&= 1 + \frac{1}{(1+x)}\sum\limits_{d=1}^\infty (y(1+x))^{(v-2)d}\\
&= 1 - \frac{1}{(1+x)} + \frac{1}{(1+x)(1-y(1+x)^{v-2})}\\
&= \frac{x+yx(1+x)^{v-2}+1}{(1+x)(1-y(1+x)^{v-2}).}
\end{align*}
\end{center}
\end{proof}

\begin{corollary}
$f^{231,312}_{v,d}(1,y)\Bigr|_{y^d} = \left|\mathcal{D}_{v,d}(231, 312)\right| = 2^{d(v-2)-1}$
\end{corollary}

\subsection{Avoiding $231, 321$}

\begin{lemma}\label{(231-321)lemma}
All labels that appear after $n=vd$ must be consecutive and increasing, and if $a_n \neq n$, then $a_{n} = n-1$.
\end{lemma}

 Let
$\beta_{v,j}^{d}$ be the generating function for descents in $\mathcal{D}_{v,d}(231,321)$. Recall Figure \ref{tikz2} is an example of $d-1$ full diamonds with $v$ vertices followed by an incomplete diamond with $j$ vertices for $j=1,\ldots,v-1.$

\begin{theorem} \label{av231321}
\[{f}_{v,d}^{231,321}(x,1)=\beta_{v,v}^d=\beta_{v,(v-1)}^{d}+x\displaystyle \sum\limits_{i=1}^{d-1}\beta_{v,(v-1)}^{d-i}\]
where \[\beta_{v,j}^{1}=
\begin{cases}
    1,& \text{if } j=1\\
    2^{j-1},& \text{if } j=2,\ldots,v-1\\
    2^{v-2},& \text{if } j=v
\end{cases}\]
is the generating function for descents  for $\mathcal{D}_{v,d}(231,321).$
\end{theorem}
\begin{proof}

We approach the proof similarly to that of Theorem \ref{av231} and partition our diamonds by the position of the largest element and proceed recursively. Because the proofs are very similar, we omit the details of this proof for brevity. The only differences are that since we are now avoiding $321$, we have no descents after the appearance of the largest label, and we have different initial conditions on one diamond.

\end{proof}

Table \ref{231321} is an example of using this recursive technique to find the generating function for descents in $\mathcal{D}_{5,3}(231,321).$

\begin{table}
\begin{center}
\begin{tabular}{ |p{.65cm}||p{1.7cm}|p{3.8cm}|p{4.2cm}|  }
 \hline
 \multicolumn{4}{|c|}{v=5} \\
 \hline
 d & 1 & 2 & 3 \\
 \hline
 $\beta_{5,1}^{d}$ & $1$    & $1+4x+3x^2$         & $1+13x+41x^2+37x^3+12x^4$ \\
 \hline
 $\beta_{5,2}^{d}$ & $1$    & $1+5x+6x^2$         & $1+15x+54x^2+62x^3+24x^4$ \\
 \hline
 $\beta_{5,3}^{d}$ & $1+x$  & $1+7x+13x^2+3x^3$   & $1+18x+80x^2+128x^3+73x^4+12x^5$ \\
 \hline
 $\beta_{5,4}^{d}$ & $1+3x$ & $1+10x+25x^2+12x^3$ & $1+22x+121x^2+248x^3+184x^4+48x^5$ \\
 \hline\hline
 $\beta_{5,5}^{d}$ & $1+3x$ & $1+11x+28x^2+12x^3$ & $1+24x+134x^2+273x^3+196x^4+48x^5$ \\
 \hline
\end{tabular}
\end{center}\caption{The recursive steps necessary to find the generating function for descents in $\mathcal{D}_{5,3}(231,321).$}\label{231321}
\end{table}

\begin{corollary}\label{av312321}
${f}_{v,d}^{231,321}(1,y)\Bigr|_{y^d}=\beta_{v,v}^d(1) = |\mathcal{D}_{v,d}(231,321)| = |\mathcal{D}_{v,d}(312,321)|.$
\end{corollary}

\begin{proof}
By Proposition \ref{Wilf}, $231,321$ is $d$-Wilf-equivalent to $312,321.$
\end{proof}

\begin{table}
\begin{center}
\begin{tabular}{ |p{1cm}||p{2cm}|p{2cm}|p{2cm}|  }
 \hline
 \multicolumn{4}{|c|}{v=5} \\
 \hline
 d & 1 & 2 & 3 \\
 \hline
 $\beta_{5,1}^{d}$ & $1$ & $8$  & $104$ \\
 \hline
 $\beta_{5,2}^{d}$ & $1$ & $12$ & $156$ \\
 \hline
 $\beta_{5,3}^{d}$ & $2$ & $24$ & $312$ \\
 \hline
 $\beta_{5,4}^{d}$ & $4$ & $48$ & $624$ \\
 \hline\hline
 $\beta_{5,5}^{d}$ & $4$ & $52$ & $676$ \\
 \hline
\end{tabular}
\end{center}\caption{The total number of permutations for $\mathcal{D}_{5,d}$ that avoid the patterns $(231,321)$ when $d = 1, 2, 3.$}
\end{table}



\section{Diamonds avoiding three or four patterns of length 3}\label{Sset}

There are only two nontrivial cases to examine when we avoid three patterns of length 3: $132,213,321$ and $231,312,321$.

\subsection{Diamonds avoiding $132,213,321$}

\begin{theorem} \label{av132213321}
$f^{132,213,321}_{v,d}(x,y)= \displaystyle \sum\limits_{d=1}^\infty \sum\limits_{\sigma \in \mathcal{D}(132,213,321)} y^dx^{des(\sigma)} = \frac{1-y+xy^2}{(1-y)^2}.$
\end{theorem}

\begin{proof}
Let $n=vd$ be the largest label in $d$ diamonds with $v$ vertices. Avoiding the pattern $132$ forces all labels before $n$ to be larger than all labels after. Avoiding the pattern $213$ forces all labels before $n$ to be increasing. Avoiding the pattern $321$ forces all labels after $n$ to be increasing.
This indicates that all vertices that appear before $n$ will be the consecutive numbers prior to $n$ and all vertices after $n$ will be the remaining elements ordered consecutively. A label $a_i = n$ iff $i = vs$ for some $s = 1, \ldots, d$, and there is only one arrangement for the rest of the elements. Therefore, there can only be, at most, one descent and it occurs between diamonds. So,
\begin{center}
\begin{align*}
\sum\nolimits y^dx^{des(\sigma)} &= 1 + \sum\limits_{d=1}^\infty y^d(1+(d-1)x)\\
&= \frac{1}{1-y} + xy^2\sum\limits_{d=1}^\infty (d-1)y^{d-2}\\
&= \frac{1-y+xy^2}{(1-y)^2.}
\end{align*}
\end{center}
\end{proof}

\begin{corollary}
$f^{132,213,321}_{v,d}(1,y)\rvert_{y^d} = \left|\mathcal{D}_{v,d}(132,213, 321)\right| = d$.
\end{corollary}

\subsection{Diamonds avoiding 231, 312, 321}

We will proceed by examining what changes can be made to the identity permutation while still avoiding $231, 312$, and $321$.

\begin{lemma}\label{231-312-321-1}For labels $a_{i}, a_{j}, a_{k}$ if $a_{i},a_{j}<a_{k},$ then $i<k$ or $j<k$ in order to avoid the patterns $312$ and $321.$
\end{lemma}

A \emph{swap} is when two consecutive labels from the identity permutation switch positions in the permutation. Since any permutation can be created from the identity using swaps, restricting our changes to swaps will not exclude any possibilities.
\begin{lemma}\label{231-312-321-2}
All swaps must be disjoint in order to avoid $321.$
\end{lemma}
\begin{proof}
We simply examine the cases when two swaps overlap in some way, either with two swaps executed on $3$ elements, or two overlapping swaps on $4$ elements.
\end{proof}

\begin{theorem} \label{av231312321}
The generating function for descents in $\mathcal{D}_{v,d}(231,312,321)$ is \[f_{v,d}^{231,312,321}(x) = (1+x)^{d-1}d\displaystyle\sum\limits_{k=0}^{\lfloor\frac{v-2}{2}\rfloor}{v-2-k \choose k}x^k.\]
\end{theorem}
\begin{proof}
Every final element of a diamond can either remain unchanged or be swapped with the least element of the next diamond. This then gives the generating function $(1+x)^{d-1}$ for each possible swap. Let $k$ represent the nonconsecutive positions from which to choose a swap among the interior vertices. Note that in a diamond there are $v-3$ positions to swap since there are $v-2$ interior vertices.
  By Lemma \ref{231-312-321-1} and Lemma \ref{231-312-321-2} any consecutive interior vertices can only be swapped disjointly. Since the swaps must be nonconsecutive, $k$ must be chosen from $v-3-(k-1).$ This gives ${v-2-k \choose k}.$ We then sum over all $k$ in order to generate all possible descents for a single diamond. Since we have $d$ diamonds in which to execute these swaps, we raise to the $d^{th}$ power. The gfd for $\mathcal{D}_{v,d}(231,312,321)$ is then $(1+x)^{d-1}\big(\Sigma_{k=0}^{\lfloor\frac{v-2}{2}\rfloor}{v-2-k \choose k}x^k\big)^d.$
\end{proof}
\begin{corollary}
${f}_{v,d}^{231,312,321}(1,y)\Bigr|_{y^d}=|\mathcal{D}_{v,d}(231,312,321)| = 2\cdot4^{d-1}$.
\end{corollary}
\begin{corollary}
$\{|\mathcal{D}_{v,1}(231,312,321)|\}_{v\geq 1}$ is the Fibonacci numbers.
\end{corollary}
\begin{proof}
The base cases are $|\mathcal{D}_{3,1}(231,312,321)|=1$ which is the identity, and $|\mathcal{D}_{4,1}(231,312,321)|=2,$ which is the identity permutation and the permutation with the interior vertices swapped. Let there be a single diamond with $v$ vertices, where are $v-2$ interior vertices that can be swapped which gives the the $\mathcal{D}_{v,1}(231,312,321)$ permutations that avoid the three patterns. Now consider a single diamond with $v+1$ vertices. The final interior vertex will either be the $v$ element when there is no descent in the last two interior vertices, or the $v-1$ element when there is a descent between the final two interior vertices. When $v$ is the final interior vertex, there are $v-2$ vertices that can be re-arranged. Thus there are the $\mathcal{D}_{v,1}(231,312,321)$ permutations. When there is a descent in the final two interior vertices, there are $v-3$ interior vertices that can be re-arranged, thus there are the $\mathcal{D}_{v-1,1}(231,312,321)$ permutations. Hence $|\mathcal{D}_{v+1,1}(231,312,321)|=|\mathcal{D}_{v,1}(231,312,321)|+|\mathcal{D}_{v-1,1}(231,312,321)|.$ Therefore, $|\mathcal{D}_{v,1}(231,312,321)|$ follows the Fibonacci numbers.
\end{proof}

\subsection{Diamonds avoiding four patterns of length 3}

\begin{theorem}
Let $S$ be a set of at least $4$ distinct permutations of length $3$. \\
Then $\left|\mathcal{D}_{v,d}(S)\right| = \begin{cases} 0, & \mbox{if}\phantom{.} 123 \in S\\ 1, & \mbox{if} \phantom{.} 123 \not \in S \end{cases}.$
\end{theorem}

\begin{proof}
Let $n$ be the largest label in any permutation. Due to the structure of the diamonds, any set of permutations involving $123$ cannot be avoided. For any other collection of $4$ or more patterns, the result is easily seen using the lemmas for avoiding a single pattern earlier in the paper.
\end{proof}

\section{Open problems}\label{open}

This investigation leaves several directions open for future study. We did not touch on patterns of length $4$, they all remain open. We are confident the techniques of Bevan et.al. \cite{Shrubs} will give the growth rate and minimal polynomial for diamonds avoiding $321$, but in addition it is likely that these techniques would also work for some patterns of length $4$. Although the minimal polynomials are unlikely to generalize, the transition operators in particular cases could potentially even generalize to length $k$ for the decreasing pattern $k\phantom{.}k-1\ldots2\phantom{.}1$. There are also a wide variety of other poset classes that could be approached in this manner other than diamonds. We generalized our diamonds by adding additional elements and order relations between the least and greatest elements, but one could also imagine creating a diamond-type poset with more than $3$ levels as another generalization.  

\acknowledgements
The authors would like to thank both referees for their careful readings of the paper with many constructive suggestions that not only improved the paper but also gave the authors excellent examples of the mathematical publishing process and language and notation standards. The authors would also like to thank David Bevan for confirming their suspicion that the techniques used in previous papers would also be successful here for $321$ avoiders and for calculating the first $250$ terms of the sequence.

\end{document}